\documentclass[10pt]{article}

\usepackage{amsmath,amsthm}
\usepackage{amssymb}
\usepackage{hyperref}

\allowdisplaybreaks
 
\begin{document}

\title{Summation formulas of $q$-hyperharmonic numbers}
\author{
Takao Komatsu 
\\
\small Department of Mathematical Sciences, School of Science\\
\small Zhejiang Sci-Tech University\\
\small Hangzhou 310018 China\\
\small \texttt{komatsu@zstu.edu.cn}\\\\
Rusen Li\\
\small School of Mathematics\\
\small Shandong University\\
\small Jinan 250100 China\\
\small \texttt{limanjiashe@163.com}
}

\date{
}

\maketitle

\def\stf#1#2{\left[#1\atop#2\right]}
\def\sts#1#2{\left\{#1\atop#2\right\}}
\def\e{\mathfrak e}
\def\f{\mathfrak f}

\newtheorem{theorem}{Theorem}
\newtheorem{Prop}{Proposition}
\newtheorem{Cor}{Corollary}
\newtheorem{Lem}{Lemma}
\newtheorem{Example}{Example}
\newtheorem{Remark}{Remark}

\begin{abstract}
In this paper, several weighted summation formulas of $q$-hyperharmonic numbers are derived. As special cases, several formulas of hyperharmonic numbers of type $\sum_{\ell=1}^{n} {\ell}^{p} H_{\ell}^{(r)}$ and $\sum_{\ell=0}^{n} {\ell}^{p} H_{n-\ell}^{(r)}$ are obtained.  
\end{abstract}

\section{Introduction}

Spie\ss {\cite{Spiess}} gives some identities including the types of $\sum_{\ell=1}^n\ell^k H_\ell$, $\sum_{\ell=1}^n\ell^k H_{n-\ell}$ and $\sum_{\ell=1}^n\ell^k H_\ell H_{n-\ell}$.  In particular, explicit forms for $r=0,1,2,3$ are given.  
In this paper, several identities including $\sum_{\ell=1}^n\ell^k H_\ell^{(r)}$ and $\sum_{\ell=1}^n\ell^k H_{n-\ell}^{(r)}$ are shown as special cases of more general results, where $H_\ell^{(r)}$ denotes hyperharmonic numbers defined in (\ref{def:hyperharmonic}). When $r=1$, $H_n=H_n^{(1)}$ is the original harmonic number defined by $H_n=\sum_{j=1}^n 1/j$.  
This paper is also motivated from the summation $\sum_{\ell=1}^n\ell^k$, which is related to Bernoulli numbers.  In \cite{Adam}, Stirling numbers are represented via harmonic numbers and hypergeometric functions related to Euler sums. In this paper, the sums involving harmonic numbers and their $q$-generalizations are expressed by using Stirling numbers and their $q$-generalizations. 

There are many generalizations of harmonic numbers.  Furthermore, some $q$-generalizations of hyperharmonic numbers have been proposed. In this paper, based upon a certain type of $q$-harmonic numbers $H_n^{(r)}(q)$ defined in (\ref{def:qhyperharmonic}), several formulas of $q$-hyperharmonic numbers are also derived as $q$-generalizations.  These results are also motivated from the $q$-analogues of the sums of consecutive integers (\cite{GZ,Schlosser,Warnaar}).   


In order to consider the weighted summations, we are motivated by the fact that the sum of powers of consecutive integers $1^k+2^k+\cdots+n^k$ can be explicitly expressed in terms of Bernoulli numbers or Bernoulli polynomials. After seeing the sums of powers for small $k$:
$$
\sum_{\ell=1}^n\ell=\frac{n(n+1)}{2},~ \sum_{\ell=1}^n\ell^2=\frac{n(n+1)(2 n+1)}{6},~ \sum_{\ell=1}^n\ell^3=\left(\frac{n(n+1)}{2}\right)^2,~ \dots\,,
$$
the formula can be written as
\begin{align}
\sum_{\ell=1}^{n}\ell^{k}&=\frac{1}{k+1}\sum_{j=0}^k\binom{k+1}{j}B_j n^{k+1-j}\label{ber}\\
                       &=\frac{1}{k+1}(B_{k+1}(n+1)-B_{k+1}(1))\quad\hbox{\cite{CFZ}}\,,
\label{ber1}
\end{align}
where Bernoulli numbers $B_n$ are determined by the recurrence formula
$$
\sum_{j=0}^k\binom{k+1}{j}B_j=k+1\quad (k\ge 0)
$$
or by the generating function
$$
\frac{t}{1-e^{-t}}=\sum_{n=0}^\infty B_n\frac{t^n}{n!}\,,
$$
and  Bernoulli polynomials $B_n(x)$ are defined by the following generating function
$$
\frac{te^{xt}}{e^{t}-1}=\sum_{n=0}^\infty B_n(x)\frac{t^n}{n!}\,.
$$

If Bernoulli numbers $\mathfrak B_n$ are defined by
$$
\frac{t}{e^t-1}=\sum_{n=0}^\infty\mathfrak B_n\frac{t^n}{n!}\,,
$$
we can see that $B_n=(-1)^n\mathfrak B_n$.  Then
$$
\sum_{\ell=1}^{n}\ell^k=\frac{1}{k+1}\sum_{j=0}^k\binom{k+1}{j}(-1)^{j}\mathfrak B_j n^{k+1-j}\,.
$$
We recall the well-known Abel's identity, which is frequently used in the present paper.
\begin{Lem}{(Abel's identity)}
For any positive integer $n$,
$$
\sum_{\ell=1}^{n} a_{\ell}b_{\ell}=s_{n}b_{n}+\sum_{\ell=1}^{n-1} s_{\ell}(b_{\ell}-b_{\ell+1})\,.
\label{abel}
$$

where
$$
s_{n}=\sum_{\ell=1}^{n}a_{\ell}.
$$
\end{Lem}

In the weight of harmonic numbers $H_n$, we have the following formulas.

\begin{Prop}
For $n,k\ge 1$,
$$
\sum_{\ell=1}^n\ell^k H_\ell=\frac{H_n}{k+1}\sum_{j=0}^k\binom{k+1}{j}B_j n^{k+1-j}
-\sum_{\ell=1}^{n-1}(H_n-H_\ell)\ell^k\,.
$$
\end{Prop}
\begin{proof}
Set $a_{\ell}={\ell^k}$ and $b_{\ell}={H_\ell}$ in Lemma \ref{abel}. With
\begin{align*}
\sum_{\ell=1}^{n-1} s_{\ell}(H_{\ell}-H_{\ell+1})&=s_{1}(H_{1}-H_{2})+\cdots +s_{n-1}(H_{n-1}-H_{n})\\
&=1^kH_{1}+\cdots +{(n-1)}^{k}H_{n-1}-s_{n-1}H_{n}\\
&=-\sum_{\ell=1}^{n-1}(H_n-H_{\ell})\ell^k\,,
\end{align*}
formula (\ref{ber}) gives the result.
\end{proof}

\begin{Prop}
For $n,k\ge 1$,
$$
\sum_{\ell=1}^n\ell^k H_\ell=\frac{H_n}{k+1}(B_{k+1}(n+1)-B_{k+1}(1))
-\sum_{\ell=1}^{n-1}\frac{B_{k+1}(\ell+1)-B_{k+1}(1)}{(k+1)(\ell+1)}\,.
$$
\end{Prop}
\begin{proof}
Set $a_{\ell}={\ell^k}$ and $b_{\ell}={H_\ell}$ in 
Lemma \ref{abel}. Formula (\ref{ber1}) gives the result.
\end{proof}

\section{Weighted summations of $q$-hyperharmonic numbers}

Many types of $q$-generalizations have been studied for harmonic numbers (e.g.,\cite{KL,WG}). In this paper, a $q$-hyperharmonic number $H_n^{(r)}(q)$ (see  \cite{MS}) is defined by
\begin{equation}
H_n^{(r)}(q)=\sum_{j=1}^n q^j H_j^{(r-1)}(q)\quad(r,n\ge 1)
\label{def:qhyperharmonic}
\end{equation}
with
$$
H_n^{(0)}(q)=\frac{1}{q [n]_q}
$$
and
$$
[n]_q=\frac{1-q^n}{1-q}\,.
$$
Note that
$$
\lim_{q\to 1}[n]_q=n\,.
$$
In this $q$-generalization,
$$ 
H_n(q)=H_n^{(1)}(q)=\sum_{j=1}^n\frac{q^{j-1}}{[j]_q}
$$ 
is a $q$-harmonic number.  When $q\to 1$, $H_n=\lim_{q\to 1}H_n(q)$ is the original harmonic number and $H_n^{(r)}=\lim_{q\to 1}H_n^{(r)}(q)$ is the $r$-th order hyperharmonic number, defined by
\begin{equation}
H_n^{(r)}=\sum_{\ell=1}^n H_\ell^{(r-1)}\quad\hbox{with}\quad H_n^{(1)}=H_n\,.
\label{def:hyperharmonic}  
\end{equation}

Mansour and Shattuck \cite[Identity 4.1, Proposition 3.1]{MS} give the following identities
\begin{align}
H_n^{(r)}(q)&=\binom{n+r-1}{r-1}_q\bigl(H_{n+r-1}(q)-H_{r-1}(q)\bigr)
\label{ph01}\\
&=\sum_{j=1}^n\binom{n+r-j-1}{r-1}_q\frac{q^{r j-1}}{[j]_q},
\label{ph02}
\end{align}
where
$$
\binom{n}{k}_q=\frac{[n]_q!}{[k]_q![n-k]_q!}
$$
is a $q$-binomial coefficient with $q$-factorials $[n]_q!=[n]_q[n-1]_q\cdots[1]_q$.
Note that the identities (\ref{ph01}) and (\ref{ph02}) are $q$-generalization of the identities (\ref{h01}) and (\ref{h02}), respectively.  
\begin{align}
H_n^{(r)}&=\binom{n+r-1}{r-1}(H_{n+r-1}-H_{r-1})\quad\hbox{\cite{Conway}}
\label{h01}\\
&=\sum_{j=1}^n\binom{n+r-j-1}{r-1}\frac{1}{j}\quad\hbox{\cite{BGG}}\,. 
\label{h02}
\end{align} 
So, we can see the recurrence relation 
for $r\ge 1$
$$ 
H_n^{(r+1)}=\frac{n+r}{r}H_n^{(r)}-\frac{1}{r}\binom{n+r-1}{r}\,.
$$

The generating function of this type of $q$-hyperharmonic numbers is given by
\begin{equation}
\sum_{n=1}^\infty H_n^{(r)}(q)z^n=\frac{-\log_q(1-q^r z)}{q(z;q)_r}\quad(r\ge 0)
\label{gen:qhyperharmo}
\end{equation}
(\cite[Theorem 3.2]{MS}), where
$$
-\log_q(1-t)=\sum_{m=1}^\infty\frac{t^m}{[m]_q}
$$
is the $q$-logarithm function and
$$
(z;q)_k:=\prod_{j=0}^{k-1}(1-z q^j)
$$
is the $q$-Pochhammer symbol.   When $q\to 1$, (\ref{gen:qhyperharmo}) is reduced to the generating function of hyperharmonic numbers:
$$
\sum_{n=1}^\infty H_n^{(r)}z^n=\frac{-\log(1-z)}{(1-z)^r}\quad(r\ge 0)\,.
$$
In fact, the same form is given by Knuth \cite{Knuth} as
$$
\sum_{n=r-1}^\infty\binom{n}{r-1}(H_n-H_{r-1})z^{n-r+1}=\frac{-\log(1-z)}{(1-z)^r}\quad(r\ge 0)\,.
$$

By (\ref{ph01}), we have
$$
H_n^{(r+1)}(q)-\frac{[n+r]_q}{[r]_q}H_n^{(r)}(q)
=-\frac{q^{r-1}}{[r]_q}\binom{n+r-1}{r}_q\,.
$$
Hence, 
\begin{equation}
H_n^{(r+1)}(q)=\frac{[n+r]_q}{[r]_q}H_n^{(r)}(q)-\frac{q^{r-1}}{[r]_q}\binom{n+r-1}{r}_q\,.
\label{qh11}
\end{equation}
By replacing $n$ by $n+1$ and $r$ by $r-1$ in (\ref{qh11}), together with the definition in (\ref{def:qhyperharmonic}), we have
\begin{align}
&[n+r]_q H_n^{(r)}(q)\notag\\
&=[n+1]_q H_{n+1}^{(r)}(q)-q^{n+r-1}\binom{n+r-1}{r-1}_q\,.
\label{qh13}
\end{align}


Mansour and Shattuck {\cite[Theorem 3.3]{MS}} also give the following formula,
\begin{equation}
H_n^{(r)}(q)=\sum_{j=1}^n q^{j(r-m)}\binom{n+r-m-j-1}{r-m-1}H_j^{(m)}(q).
\label{ph022}
\end{equation}
When $q\to 1$, (\ref{ph022}) is reduced to  
$$ 
H_n^{(r)}=\sum_{j=1}^n\binom{n+r-m-j-1}{r-m-1}H_j^{(m)}
$$ 
(see also \cite{BGG},\cite[2.4.Theorem]{BS}). When $m=0$, (\ref{ph022}) is reduced to (\ref{ph02}).

We prove a more general result of (\ref{ph01}).

\begin{theorem}
For nonnegative integers $n$ and $k$ and a positive integer $r$, we have
$$
\binom{k+r-1}{k}_q H_n^{(k+r)}(q)=\binom{n+k}{n}_q H_{n+k}^{(r)}(q)-\binom{n+k+r-1}{n}_q H_k^{(r)}(q)\,.
$$
\label{gph01}
\end{theorem}

\noindent
{\it Remark.}
If $r=1$ and $k$ is replaced by $r-1$ in Theorem \ref{gph01}, we have the identity (\ref{ph01}).
If $q\to 1$ in Theorem \ref{gph01}, we have the version of the original hyperharmonic numbers in \cite[Theorem 1]{MD}.

\begin{proof}[Proof of Theorem \ref{gph01}.]
The proof is done by induction on $k$.  When $k=0$, the identity is clear since both sides are equal to $H_n^{(r)}(q)$.  Assume, then, that the identity has been proved for $0,1,\cdots,k$. We give some explanations for the following calculation. Firstly, by replacing $r$ by $k+r$ in (\ref{qh11}), we get the first identity. Secondly, by using the inductive assumption, we get the second identity. Thirdly, by replacing $n$ by $n+k$ and $n$ by $k$ respectively in (\ref{qh13}), we get the third identity. Then, we have
\begin{align*}
&\binom{k+r}{k+1}_q H_n^{(k+r+1)}(q)\\
&=\binom{k+r}{k+1}_q\frac{[n+k+r]_q}{[k+r]_q}H_n^{(k+r)}(q)-\binom{k+r}{k+1}_q\frac{q^{k+r-1}}{[k+r]_q}\binom{n+k+r-1}{k+r}_q\\
&=\frac{[n+k+r]_q}{[k+1]_q}\binom{n+k}{n}_q H_{n+k}^{(r)}(q)-\frac{[n+k+r]_q}{[k+1]_q}\binom{n+k+r-1}{n}_q H_k^{(r)}(q)\\
&\quad -\frac{q^{k+r-1}}{[k+r]_q}\binom{k+r}{k+1}_q\binom{n+k+r-1}{k+r}_q\\
&=\frac{[n+k+1]_q}{[k+1]_q}\binom{n+k}{n}_q H_{n+k+1}^{(r)}(q)-\frac{q^{n+k+r-1}}{[k+1]_q}\binom{n+k}{n}_q\binom{n+k+r-1}{r-1}_q\\
&\quad -\frac{[n+k+r]_q}{[k+1]_q}\binom{n+k+r-1}{n}_q\frac{[k+1]_q}{[k+r]_q}H_{k+1}^{(r)}(q)\\
&\quad +\frac{[n+k+r]_q}{[k+1]_q}\binom{n+k+r-1}{n}_q\frac{q^{k+r-1}}{[k+r]_q}\binom{k+r-1}{r-1}_q\\
&\quad -\frac{q^{k+r-1}}{[k+r]_q}\binom{k+r}{k+1}_q\binom{n+k+r-1}{k+r}_q\\
&=\binom{n+k+1}{n}_q H_{n+k+1}^{(r)}(q)-\binom{n+k+r}{n}_q H_{k+1}^{(r)}(q)\,.
\end{align*}
We used the relation $[n+k+r]_q-q^n[k+r]_q=[n]_q$ in the final part.
\end{proof}

Cereceda \cite{Cereceda} gives the following formula,
$$
\lim_{n\to\infty}\frac{H_{n+1}^{(n+1)}}{H_n^{(n)}}=4\,.
$$
However, the ratio of $q$-hyperharmonic numbers of type $H_n^{(n)}(q)$ has a different phenomenon.

\begin{Prop}
For $|q|<1$, we have
$$
\lim_{n\to\infty}\frac{H_{n+1}^{(n+1)}(q)}{H_n^{(n)}(q)}=q\,.
$$
\label{ratio:ph01}
\end{Prop}

\begin{proof}
Since
$$
\frac{(1-q^{2 n+1})(1-q^{2 n})}{(1-q^{n+1})(1-q^{n})}\to 1\quad(|q|<1,~n\to\infty)
$$
and
\begin{align*}
\frac{H_{2 n+1}(q)-H_n(q)}{H_{2 n-1}(q)-H_{n-1}(q)}
&=\frac{q\left(\dfrac{1}{[n+1]_q}+\dfrac{q}{[n+2]_q}+\cdots+\dfrac{q^n}{[2 n+1]_q}\right)}{\dfrac{1}{[n]_q}+\dfrac{q}{[n+1]_q}+\cdots+\dfrac{q^{n-1}}{[2 n-1]_q}}\\
&\to q\quad(n\to\infty)\,,
\end{align*}
from (\ref{ph01}),
\begin{align*}
\frac{H_{n+1}^{(n+1)}(q)}{H_n^{(n)}(q)}&=\frac{\binom{2 n+1}{n}_q\bigl(H_{2 n+1}(q)-H_n(q)\bigr)}{\binom{2 n-1}{n-1}_q\bigl(H_{2 n-1}(q)-H_{n-1}(q)\bigr)}\\
&\to 1\cdot q=q\,.
\end{align*}
\end{proof}

\begin{theorem}
For positive integers $n$ and $r$,
\begin{align}
\sum_{\ell=1}^n q^{\ell-1}[\ell]_q H_\ell^{(r)}(q)&=\frac{[n]_q[n+r]_q}{[r+1]_q}H_n^{(r)}(q)-\frac{q^r[n-1]_q[n]_q}{([r+1]_q)^2}\binom{n+r-1}{r-1}_q\notag\\
&=\frac{[n]_q[r]_q}{[r+1]_q}H_n^{(r+1)}(q)+\frac{q^{r-1}}{[r+1]_q}\binom{n+r}{r+1}_q\,.
\label{eq:hq1}
\end{align}
\label{th:q110}
\end{theorem}

\begin{proof}
Set $a_{\ell}=q^{\ell-1}\binom{\ell+r-1}{r}_q$ and $b_{\ell}=H_{\ell+r-1}(q)$. By using Lemma \ref{abel}, we have
\begin{align}
&\sum_{\ell=1}^{n}q^{\ell-1}\binom{\ell+r-1}{r}_q H_{\ell+r-1}(q)\notag\\ 
&=\sum_{\ell=1}^{n}q^{\ell-1}\binom{\ell+r-1}{r}_q H_{n+r-1}(q)-\sum_{\ell=1}^{n-1}\frac{q^{\ell+r-1}}{[\ell+r]_q}\binom{\ell+r}{r+1}_q\notag\\
&=\binom{n+r}{r+1}_q H_{n+r-1}(q)-\frac{q^r}{[r+1]_q}\binom{n+r-1}{r+1}_q\,.
\label{ph04}
\end{align}
Hence, 
\begin{align}
&\sum_{\ell=1}^{n}q^{\ell-1}{[\ell]_q} H_{\ell}^{(r)}(q)\notag\\ 
&=\sum_{\ell=1}^{n}q^{\ell-1}{[\ell]_q} \binom{\ell+r-1}{r-1}_q (H_{\ell+r-1}(q)-H_{r-1}(q))\notag\\
&=[r]_q \sum_{\ell=1}^{n} q^{\ell-1} \binom{\ell+r-1}{r}_q (H_{\ell+r-1}(q)-H_{r-1}(q))\notag\\
&=[r]_q \sum_{\ell=1}^{n}  \binom{\ell+r-1}{r}_q H_{\ell+r-1}(q)-[r]_q H_{r-1}(q)\binom{n+r}{r+1}_q\,.
\label{ph05}
\end{align}
With the help of (\ref{ph01}), (\ref{ph04}) and (\ref{ph05}), we get the desired result.
\end{proof}

When $q\to 1$, Theorem \ref{th:q110} is reduced to the following.    

\begin{Cor}  
For $n,r\ge 1$,
\begin{align*}
\sum_{\ell=1}^{n} \ell H_{\ell}^{(r)}&=\frac{n(n+r)}{r+1}H_n^{(r)}-\frac{(n-1)^{(r+1)}}{(r-1)!(r+1)^2}\\
&=\frac{n r}{r+1}H_n^{(r+1)}+\frac{1}{r+1}\binom{n+r}{r+1}\,,
\end{align*}
where $(x)^{(n)}=x(x+1)\cdots(x+n-1)$ ($n\ge 1$) denotes the rising factorial with $(x)^{(0)}=1$. 
\label{th:110}
\end{Cor}

In order to establish similarly structured theorems of $q$-hyperharmonic numbers, we recall the $q$-Stirling numbers of the second kind, denoted by $S_{q}(n,m)$,  defined by Carlitz (see e.g. \cite{Carlitz}) as 
\begin{align}
([x]_{q})^{n}=\sum_{m=0}^{n} q^{\binom{m}{2}} S_{q}(n,m) ([x]_q)_{(m)}, \quad  (n \in \mathbb N),
\label{qstirling}
\end{align}
where $([x]_q)_{(m)}=[x]_q [x-1]_q \cdots [x-m+1]_q$ denotes the $q$-falling factorial with $([x]_q)_{0}=1$. The $q$-Stirling numbers of the second kind $S_q(n, m)$ satisfy the recurrence relation
$$
S_q(n+1,m)=S_q(n,m-1)+ [m]_q \cdot S_q(n,m)
$$
with boundary values
$$
S_q(n,0)=S_q(0,n)=\delta_{n 0}, \quad  (n \geq 0)
$$
(\cite{Gould}).

We need a $q$-version of the relation by Spie\ss {\cite{Spiess}}, which is essential in the proof of the following structured theorem of $q$-hyperharmonic numbers of type 
$\sum_{\ell=0}^{n} q^{\ell-1} {([\ell]_q)}^{p} H_{\ell}^{(r)}(q)$.

\begin{Lem}
Given summation formulas $\sum_{\ell=0}^{n} q^{\ell-1} \binom{\ell}{j}_q[c_{\ell}]_q=F_q(n,j)$ for $n,j\in\mathbb N$, one has 
$$
\sum_{\ell=0}^{n} q^{\ell-1} {([\ell]_q)}^{p} [c_{\ell}]_q=\sum_{\ell=0}^{p} q^{\binom{\ell}{2}} S_q(p,\ell) \cdot {[\ell]_q}! \cdot F_q(n,\ell)\,.
\label{qspiesslemma}
$$
where $S_q(p,\ell)$ denote the $q$-Stirling numbers of the second kind.
\end{Lem}
\begin{proof}
Using (\ref{qstirling}), we have
\begin{align*}
\sum_{\ell=0}^{n} q^{\ell-1} {([\ell]_q)}^{p} [c_{\ell}]_q
&=\sum_{\ell=0}^{n} q^{\ell-1} [c_{\ell}]_q \sum_{j=0}^{p} q^{\binom{j}{2}} S_{q}(p,j) \cdot ([\ell]_q)_{(j)}\\
&=\sum_{j=0}^{p} q^{\binom{j}{2}} S_{q}(p,j) {[j]_q}! \sum_{\ell=0}^{n} q^{\ell-1} \binom{\ell}{j}_q [c_{\ell}]_q \\
&=\sum_{j=0}^{p} q^{\binom{j}{2}} S_{q}(p,j) \cdot {[j]_q}! \cdot F_q(n,j)\,.
\end{align*}
\end{proof}

We introduce some notations. 
For $n,r, p \in \mathbb N$, set
\begin{align}
\sum_{\ell=0}^{n} q^{\ell-1} {[\ell]_q}^{p} H_{\ell}^{(r)}(q)=A_q(p,r,n)H_{n}^{(r)}(q)-B_q(p,r,n)\,.\notag
\end{align}  
From (\ref{qh11}), for $p=0$, $A_q(0,r,n)=\frac{[n+r]_q}{[r]_q}, B_q(0,r,n)=\frac{q^{r-1}}{[r]_q}\binom{n+r-1}{r}_q$. 
From Theorem \ref{th:q110}, for $p=1$, we know that 
\begin{align*}
&A_q(1,r,n)=\frac{[n]_q[n+r]_q}{[r+1]_q},\\
&B_q(1,r,n)=\frac{q^r[n-1]_q[n]_q}{([r+1]_q)^2}\binom{n+r-1}{r-1}_q\,.
\end{align*}

\begin{theorem}\label{qhypersturc1}
For $n,r,p\ge 1$,
$$ 
\sum_{\ell=0}^{n} q^{\ell-1} {[\ell]_q}^{p} H_{\ell}^{(r)}(q)=A_q(p,r,n)H_{n}^{(r)}(q)-B_q(p,r,n)\,,
$$ 
where 
\begin{align*}
&A_q(p,r,n)\\
&=\sum_{\ell=0}^{p} q^{\binom{\ell}{2}+p-1} S_q(p,\ell){[\ell]_q!}  \binom{n+r-1}{r-1}_q^{-1} \binom{r+\ell-1}{\ell}_q \binom{r+n}{r+\ell}_q,\\
&B_q(p,r,n)=\sum_{\ell=0}^{p} \frac{q^{\binom{\ell}{2}+r+2p-2}}{[r+\ell]_q}S_q(p,\ell){[\ell]_q!}\binom{r+\ell-1}{\ell}_q\binom{r+n-1}{r+\ell}_q\,.
\end{align*}
\end{theorem}
\begin{proof}
Set $[c_{\ell}]_q=H_{\ell}^{(r)}(q)$ in Lemma \ref{qspiesslemma}. Then by using Lemma \ref{abel}, we have
\begin{align}
F_q(n,p)&=\sum_{\ell=0}^{n} q^{\ell-1} \binom{\ell}{p}_{q} H_{\ell}^{(r)}(q) \notag\\
&=\sum_{\ell=1}^{n} q^{\ell-1} \binom{\ell}{p}_{q} \binom{\ell+r-1}{r-1}_{q} (H_{\ell+r-1}(q)-H_{r-1}(q))\notag\\
&=\sum_{\ell=1}^{n} q^{\ell-1} \binom{r+p-1}{p}_{q} \binom{\ell+r-1}{r+p-1}_{q} (H_{\ell+r-1}(q)-H_{r-1}(q))\notag\\
&=q^{p-1} \binom{r+p-1}{p}_q \binom{r+n}{r+p}_q H_{n+r-1}(q)\notag\\
&\quad -q^{p+r-1} \binom{r+p-1}{p}_q \sum_{\ell=1}^{n-1} \frac{q^{\ell-1}}{[\ell+r]_q}\binom{r+\ell}{r+p}_q\notag\\
&\quad \quad-\binom{r+p-1}{p} \binom{r+n}{r+p} H_{r-1}\notag\\
&=q^{p-1} \binom{r+p-1}{p}_q \binom{r+n}{r+p}_q (H_{n+r-1}(q)-H_{r-1}(q))\notag\\
&\quad \quad -\frac{q^{r+2p-2}}{[r+p]_q}\binom{r+p-1}{p}_q \binom{r+n-1}{r+p}_q. \label{ph06}
\end{align}
With the help of (\ref{ph01}) and (\ref{ph06}), Lemma \ref{qspiesslemma} gives the result.
\end{proof}

When $q\to 1$, Theorem \ref{qhypersturc1} is reduced to the following.  

\begin{Cor}\label{hypersturc1}
For $n,r,p\ge 1$,
$$ 
\sum_{\ell=0}^{n} {\ell}^{p} H_{\ell}^{(r)}=A(p,r,n)H_{n}^{(r)}-B(p,r,n)\,,
$$ 
where \begin{align*}
&A(p,r,n)=\sum_{\ell=0}^{p} S(p,\ell){\ell!}{\binom{n+r-1}{r-1}^{-1}\binom{r+\ell-1}{\ell}\binom{r+n}{r+\ell}},\\
&B(p,r,n)=\sum_{\ell=0}^{p} \frac{1}{r+\ell}S(p,\ell){\ell!}\binom{r+\ell-1}{\ell}\binom{r+n-1}{r+\ell}\,.
\end{align*}
\end{Cor} 

\begin{Example}
$p=2$ gives
\begin{align}
&\sum_{\ell=1}^n q^{\ell-1}([\ell]_q)^2 H_\ell^{(r)}(q)\notag\\
&=\frac{[n]_q[n+r]_q(1+q[r+1]_q[n]_q)}{[r+1]_q[r+2]_q}H_n^{(r)}(q)\notag\\
&\quad-q^{r}[n-1]_q[n]_q\binom{n+r-1}{r-1}_q\frac{q[r+1]_q^2[n]_q-q^3[r]_q^2+[2]_q}{[r+1]_q^2[r+2]_q^2}\,.
\label{ph08}
\end{align}
\end{Example}

Note that $[\ell+1]_q=1+q \cdot [\ell]_q$ and $[\ell+2]_q=[2]_q+q^2 \cdot [\ell]_q$. With the help of Theorem \ref{qhypersturc1} and  identities (\ref{eq:hq1}) and (\ref{ph08}), we have the following identities.
For positive integers $n$ and $r$,
\begin{align*}
&\sum_{\ell=1}^n q^{\ell-1}[\ell]_q[\ell+1]_q H_\ell^{(r)}(q)\notag\\
&=\frac{[n]_q[n+r]_q([2]_q[n+2]_q+q^3[r-1]_q[n+1]_q)}{[r+1]_q[r+2]_q}H_n^{(r)}(q)\notag\\
&\quad-q^r[n-1]_q[n]_q\binom{n+r-1}{r-1}_q\frac{[2]_q[r+2]_q^2+q^4[r+1]_q^2[n-2]_q}{[r+1]_q^2[r+2]_q^2}\,.
\end{align*}\
{\scriptsize
\begin{align*}
&\sum_{\ell=1}^nq^{\ell-1}[\ell]_q[\ell+1]_q[\ell+2]_q H_\ell^{(r)}(q)\\
&=\frac{[n]_q[n+r]_q\bigl((r+1)(r+2)n^2+3(r+1)(r+4)n+2(r^2+6 r+11)\bigr)}{[r+1]_q[r+2]_q[r+3]_q}H_n^{(r)}(q)\\
&\quad -q^r[n-1]_q[n]_q\binom{n+r-1}{r-1}_q\frac{(r+1)^2(r+2)^2 n^2+(r+1)^2(r^2+16 r+34)n+12(3 r^2+12 r+11)}{([r+1]_q)^2([r+2]_q)^2([r+3]_q)^2}\,.
\end{align*}
}

To give a more general result, we need the $q$-unsigned Stirling numbers of the first kind $s_{uq}(n,k)$ defined by
$$
[\ell]_q^{(n)}= [\ell]_q [\ell+1]_q\cdots[\ell+n-1]_q=\sum_{k=0}^{n} s_{uq}(n,k) ([\ell]_q)^{k}, \quad  (n \in \mathbb N).
$$
The $q$-unsigned Stirling numbers of the first kind $s_{uq}(n,k)$ are well defined since $[\ell+m]_q=[m]_q+q^{m} \cdot [\ell]_q$. 
\begin{theorem}\label{qharmonicconsec}
For positive integers $n, p$ and $r$,
\begin{align*}
\sum_{\ell=1}^n q^{\ell-1} [\ell]_q^{(p)} H_\ell^{(r)}(q)=A_{1q}(p,r,n)H_{n}^{(r)}-B_{1q}(p,r,n)\,,
\end{align*}
where
\begin{align*}
&A_{1q}(p,r,n)=\sum_{m=0}^p  s_{uq}(p.m) A_{q}(m,r,n),\\
&B_{1q}(p,r,n)=\sum_{m=0}^p  s_{uq}(p.m) B_{q}(m,r,n).\\
\end{align*}
\end{theorem}
\begin{proof}
\begin{align*}
&\sum_{\ell=1}^n q^{\ell-1} [\ell]_q^{(p)} H_\ell^{(r)}(q)\\ 
&=\sum_{\ell=1}^n q^{\ell-1} \sum_{m=0}^{p} s_{uq}(p,m) {[\ell]_q}^{m}H_\ell^{(r)}(q)\\
&=\sum_{m=0}^{p} s_{uq}(p,m) \sum_{\ell=1}^n q^{\ell-1} {[\ell]_q}^{m}H_\ell^{(r)}(q)\\
&=\sum_{m=0}^{p} s_{uq}(p,m) (A_{q}(m,r,n)H_{n}^{(r)}(q)-B_{q}(m,r,n))\\
&=\left( \sum_{m=0}^p s_{uq}(p,m) A_{q}(m,r,n) \right) H_{n}^{(r)}(q)
-\left( \sum_{m=0}^p s_{uq}(p,m) B_{q}(m,r,n) \right).\\
\end{align*}
\end{proof}

When $q\to 1$, Theorem \ref{qharmonicconsec} is reduced to the following.   
\begin{Cor}\label{harmonicconsec}
For positive integers $n, p$ and $r$,
\begin{align*}
\sum_{\ell=1}^n (\ell)^{(p)}H_\ell^{(r)}=A_{1}(p,r,n)H_{n}^{(r)}-B_{1}(p,r,n)\,,
\end{align*}
where
\begin{align*}
A_{1}(p,r,n)=\sum_{m=0}^p (-1)^{p+m} s(p.m) A(m,r,n),\\
B_{1}(p,r,n)=\sum_{m=0}^p (-1)^{p+m} s(p.m) B(m,r,n).\\
\end{align*}
\end{Cor}

\subsection{Backward summations}

Now we consider backward summations of $q$-hyperharmonic numbers.

\begin{theorem}
For positive integers $n$ and $r$,
\begin{multline*} 
\sum_{\ell=1}^n q^{2n-2\ell} [\ell]_q H_{n-\ell}^{(r)}(q)\\
=\frac{[n]_q[n+r]_q}{[r]_q[r+1]_q}H_n^{(r)}(q)
-\binom{n+r}{r+1}_q \left(\frac{q^{r-1}}{[r]_q}+\frac{q^{r}}{[r+1]_q}-\frac{q^{n+r-1}}{[n+r]_q}\right)\,.
\end{multline*}
\label{qbhh1}
\end{theorem}

\begin{proof}
Set $a_{\ell}=q^{n-\ell}{H_{n-\ell}^{(r)}(q)}$, and $b_{\ell}=q^{n-\ell}{[\ell]_q}$. By using Lemma \ref{abel} and $[\ell+1]_q-q[\ell]_q=1$, we have
\begin{align*}
&\sum_{\ell=1}^n q^{2n-2\ell} [\ell]_q H_{n-\ell}^{(r)}(q)\\ 
&=[n]_q \cdot H_{n-1}^{(r+1)}(q)+\sum_{\ell=1}^{n-1}(H_{n-1}^{(r+1)}(q)-H_{n-\ell-1}^{(r+1)}(q))(q^{n-\ell}{[\ell]_q}-q^{n-\ell-1}{[\ell+1]_q})\notag\\
&=[n]_q \cdot H_{n-1}^{(r+1)}(q)+\sum_{\ell=1}^{n-1}H_{n-1}^{(r+1)}(q)(q^{n-\ell}{[\ell]_q}-q^{n-\ell-1}{[\ell+1]_q})\\
&\quad +\sum_{\ell=1}^{n-1}H_{n-\ell-1}^{(r+1)}(q)(-q^{n-\ell}{[\ell]_q}+q^{n-\ell-1}{[\ell+1]_q})\\
&=q^{n-1}H_{n-1}^{(r+1)}(q)+\sum_{\ell=1}^{n-1}q^{n-\ell-1}H_{n-\ell-1}^{(r+1)}(q)\\
&=H_{n-1}^{(r+2)}(q)\,.
\end{align*}
With the help of (\ref{ph01}), we get the desired result.
\end{proof}

When $q\to 1$, Theorem \ref{qbhh1} is reduced to the following. 

\begin{Cor}  
For positive integers $n$ and $r$,
$$
\sum_{\ell=1}^n\ell H_{n-\ell}^{(r)}=\frac{n(n+r)}{r(r+1)}H_n^{(r)}-\frac{(n)^{(r)}\bigl((2 r+1)n+r^2\bigr)}{(r-1)!r^2(r+1)^2}\,.
$$
\label{bhh1}
\end{Cor}

It is more complicated to get a summation formula for the backward summations of higher power.  In the case where $q\to 1$, we have more relations, including the following.      

\begin{theorem}
For positive integers $n, p$ and $r$,
\begin{align*}
\sum_{\ell=0}^{n} {\ell}^{p} H_{n-\ell}^{(r)}=A_2(p,r,n)H_{n}^{(r)}-B_2(p,r,n)\,.
\end{align*}
\label{bhh2}
where $A_2(p,r,n)$ and $B_2(p,r,n)$ satisfy the following relations:
\begin{align*}
&A_2(p,r,n)=A_2(0,r,n)\left(1+\sum_{j=0}^{p-1}\binom{p}{j}A_2(j,r+1,n-1)\right),\\
&B_2(p,r,n)\\
&=B_2(0,r,n)\left(1+\sum_{j=0}^{p-1}\binom{p}{j}A_2(j,r+1,n-1)\right)+\sum_{j=0}^{p-1}\binom{p}{j}B_2(j,r+1,n-1)\,,
\end{align*}
with the initial values $A_2(0,r,n)=\frac{n}{r}$ and $B_2(0,r,n)=\frac{1}{r}\binom{n+r-1}{r}$.
\end{theorem}

Nevertheless, we can have a different backward summation formula without weights.

\begin{theorem}
For positive integers $n, p$ and $r$,
$$
\sum_{\ell=1}^n q^{p(n-\ell)} H_{n-\ell}^{(r)}(q)=C_q(p,r,n)H_n^{(r)}(q)-D_q(p,r,n)\,,
$$
where $C_q(p,r,n)$ and $D_q(p,r,n)$ satisfy the following recurrence relation.
\begin{align*}
&C_q(p,r,n)=\frac{[n]_q}{[r]_q}\left(q^{(p-1)(n-1)}+(1-q^{p-1})C_q(p-1,r+1,n-1)\right)\\
&D_q(p,r,n)\\
&=\frac{q^{r-1}[n]_q}{([r]_q)^2}\binom{n+r-1}{r}_q \left(q^{(p-1)(n-1)}+(1-q^{p-1})C_q(p-1,r+1,n-1)\right)\\
&\quad +(1-q^{p-1})D_q(p-1,r+1,n-1).
\end{align*}
\end{theorem}

\begin{proof}
Set $a_{\ell}=q^{n-\ell}{H_{n-\ell}^{(r)}(q)}$ and $b_{\ell}=q^{(p-1)(n-\ell)}$. By using Lemma \ref{abel} and $[\ell+1]_q-q[\ell]_q=1$, we have
\begin{align*}
&\sum_{\ell=1}^n q^{p(n-\ell)} H_{n-\ell}^{(r)}(q)\\ 
&=H_{n-1}^{(r+1)}(q)+\sum_{\ell=1}^{n-1}(H_{n-1}^{(r+1)}(q)-H_{n-\ell-1}^{(r+1)}(q))(q^{(p-1)(n-\ell)}-q^{(p-1)(n-\ell-1)})\notag\\
&=H_{n-1}^{(r+1)}(q)+\sum_{\ell=1}^{n-1}H_{n-1}^{(r+1)}(q)(q^{(p-1)(n-\ell)}-q^{(p-1)(n-\ell-1)})\\
&\quad +\sum_{\ell=1}^{n-1}H_{n-\ell-1}^{(r+1)}(q)(-q^{(p-1)(n-\ell)}+q^{(p-1)(n-\ell-1)})\\
&=q^{(p-1)(n-1)}H_{n-1}^{(r+1)}(q)+(1-q^{p-1})\sum_{\ell=1}^{n-1}q^{(p-1)(n-\ell-1)}H_{n-\ell-1}^{(r+1)}(q)\\
&=q^{(p-1)(n-1)}H_{n-1}^{(r+1)}(q)\\
&\quad +(1-q^{p-1})\left(C_q(p-1,r+1,n-1)H_{n-1}^{(r+1)}(q)-D_q(p-1,r+1,n-1)\right)\,.
\end{align*}
With the help of (\ref{ph01}), we get the desired result.
\end{proof}

\section*{Acknowledgement}  

Authors are grateful to the anonymous referee for helpful comments.


\begin{thebibliography}{99}

\bibitem{Adam}  
V. Adamchik,  {\em  
On Stirling numbers and Euler sums},  
J. Comput. Appl. Math. {\bf 79} (1997), 119--130.  

\bibitem{BGG}
A. T. Benjamin, D. Gaebler and R. Gaebler, {\em
A combinatorial approach to hyperharmonic numbers},
Integers {\bf 3} (2003), A15.

\bibitem{BS}
M. Bah\c{s}i and S. Solak, {\em
An application of hyperharmonic numbers in matrices},
Hacet. J. Math. Stat. {\bf 42} (2013), 387--393.

\bibitem{Carlitz}
L. Carlitz, {\em $q$-Bernoulli numbers and polynomials},
Duke Math. J. {\bf 15} (1948), 987--1000.

\bibitem{Cereceda}
J. L. Cereceda, {\em
An introduction to hyperharmonic numbers},
Internat. J. Math. Ed. Sci. Tech. {\bf 46} (2015), 461--469.

\bibitem{CFZ}
W. Y. C. Chen, A. M. Fu and I. F. Zhang, {\em
Faulhaber's theorem on power sums},
Discrete Math. {\bf 309} (2009), 2974--2981.


\bibitem{Conway}
J. H. Conway, and R. K. Guy, {\em
The Book of Numbers},
Springer, New York (1996).

\bibitem{Gould}  
H. Gould,  {\em  
The $q$-Stirling numbers of the first and second kind},  
Duke Math. J. {\bf 28} (1961), 281--289. 

\bibitem{GZ}
V. J. W. Guo and J. Zeng,  {\em
A $q$-analogue of Faulhaber's formula for sums of powers},
Electron. J. Combin. {\bf 11}(2) (2005), \#R19. 

\bibitem{Knuth}
D. E. Knuth, {\em
The art of computer programming},
Vols. 1-3, Addison-Wesley, Reading, Mass., 1968.

\bibitem{KL}  
T. Komatsu and R. Li, {\em  
Infinite series containing generalized $q$-harmonic numbers}, 
Integers {\bf 21} (2021), Paper No. A1, 13 pp.   

\bibitem{MS}
T. Mansour and M. Shattuck,  {\em
A $q$-analog of the hyperharmonic numbers},
Afr. Mat. {\bf 160} (2014), 147--160.

\bibitem{MD}
I. Mez\H o and A. Dil,  {\em
Hyperharmonic series involving Hurwitz zeta function},
J. Number Theory {\bf 130} (2010), 360--369.

\bibitem{Schlosser}
M. Schlosser,  {\em
$q$-analogues of the sums of consecutive integers, squares, cubes, quarts and quints},
Electron. J. Combin.  {\bf 11} (2004), \#R71.

\bibitem{Spiess}
J. Spie\ss, {\em
Some identities involving harmonic numbers},
Math. Comp. {\bf 55} (1990), 839--863.

\bibitem{Warnaar}
S. O. Warnaar, {\em
On the $q$-analogue of the sum of cubes},
Electron. J. Combin. {\bf 11} (2004), \#N13.

\bibitem{WG}   
C. Wei and Q. Gu, {\em 
$q$-generalizations of a family of harmonic number identities}, 
Adv. in Appl. Math. {\bf 45} (2010), 24--27.  

\end{thebibliography}
\end{document}